\documentclass[12pt,reqno,twoside]{amsart}

\usepackage[T1]{fontenc}
\usepackage[utf8]{inputenc}
\usepackage{microtype}
\usepackage{cite}

\usepackage{amsmath,amssymb,amsthm,mathtools}
\usepackage{bm}
\usepackage{mathrsfs}
\usepackage{dsfont}

\usepackage{graphicx,tikz,caption,subcaption}
\usepackage{enumitem}

\usepackage[colorlinks=true,linkcolor=blue!50!black,citecolor=green!40!black,urlcolor=magenta!60!black]{hyperref}
\usepackage[nameinlink,capitalize]{cleveref}

\usepackage{bm}
\usepackage{amsbsy}

\numberwithin{equation}{section}

\allowdisplaybreaks
\theoremstyle{plain}
\newtheorem{Th}{Theorem}[section]

\theoremstyle{definition}
\newtheorem{Lem}[Th]{Lemma}

\newtheorem{Cor}[Th]{Corollary}

\begin{document}

\title[On the Nazarov--Shcheglova Conjecture]{On the Nazarov--Shcheglova Conjecture for Sharp Sobolev Inequalities: The Case $(n,p)=(3,2)$}
\author[M. Jleli]{Mohamed Jleli}
\address{(M. Jleli) Department of Mathematics, College of Science, King Saud University, Riyadh 11451, Saudi Arabia}
\email{jleli@ksu.edu.sa}

\author[B. Samet]{Bessem Samet}
\address{(B. Samet) Department of Mathematics, College of Science, King Saud University, Riyadh 11451, Saudi Arabia}
\email{bsamet@ksu.edu.sa}

\keywords{Sharp Sobolev inequality, optimal constant, extremal functions, Nazarov--Shcheglova conjecture}
\renewcommand{\subjclassname}{\textbf{2020 Mathematics Subject Classification}}
\subjclass[2020]{46E35, 26D10}

\begin{abstract}
For integers $n>k\geq0$ and $1\leq p,q\leq\infty$, let
$\lambda_3(n,k,p,q)$ denote the optimal constant in the one-dimensional
Sobolev inequality
\[
\|u^{(k)}\|_{L^q(0,1)}
\leq
\lambda_3(n,k,p,q)
\|u^{(n)}\|_{L^p(0,1)},
\qquad
u\in\mathring W_p^n(0,1).
\]
Nazarov and Shcheglova \cite{NazarovShcheglova} conjectured that
\[
\lambda_3(n,1,p,1)
=
2\lambda_3(n,0,p,\infty),
\qquad
n\geq2,\quad 1\leq p\leq\infty,
\]
and that the corresponding extremal functions coincide and are symmetric
about the midpoint of the interval. We prove this conjecture for
$(n,p)=(3,2)$ and, in particular, obtain
\[
\lambda_3(3,1,2,1)
=
\frac{1}{32\sqrt5}.
\]
The proof is based on a reduction to an operator norm problem with an
additional moment constraint. After rescaling to $(-1,1)$, this constraint
leads to orthogonality with respect to quadratic polynomials. We extend
the inverse of the second derivative through the corresponding orthogonal
projection and analyze the adjoint operator by an odd--even decomposition.
The odd component is controlled by a positive Gram kernel, while the even
component is treated by a weak-$L^2$ estimate. The equality cases yield
the characterization of all extremal functions.
\end{abstract}

\maketitle

\tableofcontents

\section{Introduction} 

The determination of sharp constants in Sobolev inequalities is a classical
problem in analysis, closely connected with the identification and qualitative
properties of the corresponding extremal functions. In the Euclidean setting,
the fundamental results of Aubin and Talenti \cite{Aubin,Talenti} established
the optimal constants in the classical Sobolev inequality and identified the
corresponding extremal functions. For a broader treatment of Sobolev-type inequalities and
their geometric and analytic aspects, we refer to \cite{SaloffCoste}. Another important development was the mass-transport approach of
Cordero-Erausquin, Nazaret, and Villani
\cite{CorderoErausquinNazaretVillani}, which yields sharp Sobolev and
Gagliardo--Nirenberg inequalities together with a characterization of the
equality cases.

In the one-dimensional setting, sharp Sobolev-type inequalities have a long
history and have been investigated from several points of view; see, for instance,
\cite{Richardson,Kalyabin,Nazarov2000,BennewitzSaito,Kalyabin2010,
KazimirovSheipak2024}. Beyond the determination of optimal constants, a
central question concerns the structure and symmetry of extremal functions;
see, for example,
\cite{BelloniKawohl,DacorognaGangboSubia,WatanabeEtAl,
SheipakGarmanova2019,BoultonLang}.  Classical examples include the
Poincar\'e, Wirtinger, and Steklov-type inequalities, while higher-order
Sobolev embeddings give rise to a considerably richer class of extremal
problems; see, for instance,
\cite{GarmanovaSheipak2023,GarmanovaSheipak2024,HindovEtAl,
KazimirovSheipak2025}. For surveys and further historical references, we
refer to \cite{KuznetsovNazarov,NazarovShcheglova}.

Adopting the notation of Nazarov and Shcheglova
\cite{NazarovShcheglova}, for integers $n>k\geq0$ and
$1\leq p,q\leq\infty$, let $\lambda_3(n,k,p,q)$ denote the best
constant in the higher-order Sobolev inequality
\begin{equation}\label{eq:lambda3}
\|u^{(k)}\|_{L^q(0,1)}
\leq
\lambda_3(n,k,p,q)\,
\|u^{(n)}\|_{L^p(0,1)},
\qquad
u\in\mathring W_p^n(0,1).
\end{equation}
Here
\[
\mathring W_p^n(0,1)
=
\left\{
u\in W_p^n(0,1):
u^{(j)}(0)=u^{(j)}(1)=0,\quad j=0,\ldots,n-1
\right\}.
\]
Exact values of $\lambda_3(n,k,p,q)$ and the structure of the
corresponding extremal functions are known only for restricted ranges
of the parameters; see \cite{NazarovShcheglova} and the references
therein.

Among the known results, the endpoint regimes $q=\infty$ and $q=1$
are particularly relevant to the present work. For $q=\infty$, sharp
constants and the corresponding extremal functions have been studied
in a number of cases; see, for example,
\cite{GarmanovaSheipak2023,GarmanovaSheipak2024,
KazimirovSheipak2024,KazimirovSheipak2025}.
In particular, Watanabe, Kametaka, Nagai, Yamagishi, and Takemura
\cite{WatanabeEtAl} determined $\lambda_3(3,0,p,\infty)$ for
$1<p<\infty$ and established the symmetry of the corresponding
extremal functions with respect to $x=\frac12$. Together with the
endpoint cases $p=1$ and $p=\infty$, this gives the determination of
$\lambda_3(3,0,p,\infty)$ for every $1\leq p\leq\infty$; see
\cite{NazarovShcheglova} and the references therein. At the endpoint
$p=1$, the right-hand side of \eqref{eq:lambda3} is understood in the
measure-valued sense.

For the endpoint $q=1$, a recent higher-order result was obtained by
Hindov, Nitzan, Olsen, and Rydhe \cite{HindovEtAl}. They proved that
\[
\lambda_3(n,0,2,1)
=
\frac{n!}{(2n)!\sqrt{2n+1}},
\]
and showed that the corresponding extremal functions are, up to
multiplication by a nonzero constant,
\[
x^n(1-x)^n.
\]

In \cite{NazarovShcheglova},  Nazarov and Shcheglova conjectured  that 
\[
\lambda_3(n,1,p,1)
=
2\lambda_3(n,0,p,\infty),
\qquad
n\geq2,\quad 1\leq p\leq\infty.
\]
Moreover, the corresponding extremal functions
coincide and are symmetric with respect to $x=\frac12$.
The conjecture is known to hold for $n=2$; see
\cite[Remark~4.15]{NazarovShcheglova}.

In the present paper, we prove the Nazarov--Shcheglova conjecture in the
case $(n,p)=(3,2)$. For $p=2$, the result of Watanabe et al.
\cite{WatanabeEtAl} gives
\begin{equation}\label{eq:constant-WatanabeEtAl}
\lambda_3(3,0,2,\infty)
=
\frac{1}{64\sqrt5}.
\end{equation}
Thus, in this case, the conjecture reduces to proving that
\[
\lambda_3(3,1,2,1)
=
\frac{1}{32\sqrt5},
\]
together with the asserted coincidence and symmetry of the corresponding
extremal functions.

Our main result is the following.

\begin{Th}\label{thm:main}
The sharp constant in \eqref{eq:lambda3} for
\[
(n,k,p,q)=(3,1,2,1)
\]
is given by
\begin{equation}\label{eq:sharp-cst}
\lambda_3(3,1,2,1)
=
\frac{1}{32\sqrt5}.
\end{equation}
Moreover, the families of extremal functions for
$\lambda_3(3,1,2,1)$ and $\lambda_3(3,0,2,\infty)$ coincide
up to multiplication by a nonzero constant, and every such extremal
function is symmetric about $x=\frac12$.
\end{Th}

Our proof is inspired in part by the left-inverse approach developed by
Hindov, Nitzan, Olsen, and Rydhe \cite{HindovEtAl} for the sharp embedding
of $W_0^{k,2}(-1,1)$ into $L^1(-1,1)$. In the present problem, however,
the reduction $v=u'$ introduces the additional moment condition
\[
\int_0^1 v(x)\,dx=0.
\]
After rescaling to $(-1,1)$ by
\[
V(t)=v\left(\frac{t+1}{2}\right),
\qquad -1<t<1,
\]
the boundary and moment conditions imply
\[
V''\perp\mathcal P_2,
\qquad
\mathcal P_2=\operatorname{span}\{1,t,t^2\}.
\]
Indeed, the boundary conditions yield the orthogonality of $V''$ to
$1$ and $t$, while the zero-moment condition yields the orthogonality
to $t^2$. This additional constraint prevents the fixed-sign kernel
argument used in \cite{HindovEtAl} from applying directly.

To overcome this difficulty, we extend the inverse of the second
derivative from $\mathcal P_2^\perp$ to $L^2(-1,1)$ by composing it
with the orthogonal projection onto $\mathcal P_2^\perp$. We then
analyze the adjoint of the resulting operator through an odd--even
decomposition. The odd component is controlled by an explicit positive
kernel, whereas the even component requires a weak-$L^2$ estimate.
This yields both the sharp constant and the characterization of the
extremal functions.

The rest of the paper is organized as follows. In Section~\ref{sec2},
we develop the reductions, construct the relevant inverse operator, and
establish the kernel estimates needed in the sequel. In Section~\ref{sec3},
we prove Theorem~\ref{thm:main}.

\section{Preliminaries}\label{sec2}

\subsection{Notation}

We first fix some notation. Throughout the paper, $\operatorname{sgn}$ denotes
the sign function,
\[
\operatorname{sgn}(x)
=
\begin{cases}
1, & x>0,\\
0, & x=0,\\
-1, & x<0.
\end{cases}
\]

For $s\in\mathbb R$, let
\[
s_+=\max\{s,0\}.
\]
For each $x\in(-1,1)$, define
\begin{equation}\label{eq:phi-x}
\phi_x(t)=(x-t)_+,
\qquad -1<t<1.
\end{equation}

We denote by
\[
\langle f,h\rangle_{L^2(-1,1)}
=
\int_{-1}^1 f(t)h(t)\,dt
\]
the inner product in $L^2(-1,1)$.

We introduce the space
\begin{equation}\label{eq:X}
X
=
\left\{
V\in H_0^2(-1,1):
\int_{-1}^1 V(t)\,dt=0
\right\},
\end{equation}
and set
\begin{equation}\label{eq:Cstar}
C_*
=
\sup_{V\in X\setminus\{0\}}
\frac{\|V\|_{L^1(-1,1)}}
{\|V''\|_{L^2(-1,1)}}.
\end{equation}

We set
\[
\mathcal P_2
=
\operatorname{span}\{1,t,t^2\},
\]
the space of real polynomials of degree at most two, viewed as a subspace
of $L^2(-1,1)$. We denote by
\[
\Pi_2:L^2(-1,1)\longrightarrow\mathcal P_2
\]
the orthogonal projection onto $\mathcal P_2$, and by $I$ the identity
operator on $L^2(-1,1)$.

For each $x\in(-1,1)$, define
\begin{equation}\label{eq:rx}
r_x=(I-\Pi_2)\phi_x
\in\mathcal P_2^\perp,
\end{equation}
where $\mathcal P_2^\perp$ denotes the orthogonal complement of
$\mathcal P_2$ in $L^2(-1,1)$.

For $f\in L^2(-1,1)$, set
\begin{equation}\label{eq:A}
(Af)(x)
=
\int_{-1}^1 r_x(t)f(t)\,dt,
\qquad -1<x<1.
\end{equation}

\begin{Lem}\label{lem:A-bounded}
The operator
\[
A:L^2(-1,1)\longrightarrow L^1(-1,1)
\]
defined by \eqref{eq:A} is linear and bounded.
\end{Lem}

\begin{proof}
Linearity is immediate. Since $I-\Pi_2$ is the orthogonal projection
onto $\mathcal P_2^\perp$,
\begin{equation}\label{est:norm-rx}
\|r_x\|_{L^2(-1,1)}
\leq
\|\phi_x\|_{L^2(-1,1)}
=
\left(\frac{(x+1)^3}{3}\right)^{1/2}
\leq
\sqrt{\frac83},
\qquad -1<x<1.
\end{equation}
Hence, by the Cauchy--Schwarz inequality,
\[
|(Af)(x)|
\leq
\sqrt{\frac83}\,
\|f\|_{L^2(-1,1)},
\qquad -1<x<1.
\]
Therefore,
\[
\|Af\|_{L^1(-1,1)}
\leq
2\sqrt{\frac83}\,
\|f\|_{L^2(-1,1)},
\]
and the result follows.
\end{proof}

We denote by
\[
A^*:L^\infty(-1,1)\longrightarrow L^2(-1,1)
\]
the Banach adjoint of $A$, where $L^2(-1,1)$ is identified
isometrically with its dual in the usual way. By
\eqref{est:norm-rx} and Fubini's theorem, for every
$g\in L^\infty(-1,1)$,
\begin{equation}\label{eq:A-star}
(A^*g)(y)
=
\int_{-1}^1 g(x)r_x(y)\,dx
\end{equation}
for almost every $y\in(-1,1)$.

We denote the corresponding operator norms by
\begin{equation}\label{eq:norms}
\begin{aligned}
\|A\|_{2\to1}
&=
\sup_{f\in L^2(-1,1)\setminus\{0\}}
\frac{\|Af\|_{L^1(-1,1)}}
{\|f\|_{L^2(-1,1)}},
\\
\|A^*\|_{\infty\to2}
&=
\sup_{g\in L^\infty(-1,1)\setminus\{0\}}
\frac{\|A^*g\|_{L^2(-1,1)}}
{\|g\|_{L^\infty(-1,1)}}.
\end{aligned}
\end{equation}

We introduce the reflection operators
\begin{equation}\label{eq:R2}
R_2:L^2(-1,1)\longrightarrow L^2(-1,1),
\qquad
(R_2f)(t)=f(-t),
\end{equation}
and
\begin{equation}\label{eq:Rinfty}
R_\infty:L^\infty(-1,1)\longrightarrow L^\infty(-1,1),
\qquad
(R_\infty g)(t)=g(-t),
\end{equation}
with the identities understood almost everywhere on $(-1,1)$.

Finally, define
\begin{equation}\label{eq:D2}
D^2:X\longrightarrow L^2(-1,1),
\qquad
D^2V=V''.
\end{equation}

\subsection{Reduction to an operator norm problem}

We first reduce the computation of $\lambda_3(3,1,2,1)$ to the
determination of the operator norm of $A$.

\begin{Lem}\label{lem:reduction}
The following assertions hold.
\begin{enumerate}
\item[\rm(i)]
The operator $D^2$ defined by \eqref{eq:D2} maps $X$ bijectively onto
$\mathcal P_2^\perp$, and
\[
\left.A\right|_{\mathcal P_2^\perp}
=
(D^2)^{-1}.
\]

\item[\rm(ii)]
The constant $C_*$ defined by \eqref{eq:Cstar} satisfies
\[
C_*
=
\|A\|_{2\to1}
=
\|A^*\|_{\infty\to2}.
\]

\item[\rm(iii)]
The map
\[
T:\mathring W_2^3(0,1)\longrightarrow X,
\qquad
(Tu)(t)
=
u'\left(\frac{t+1}{2}\right),
\]
is a linear bijection, and
\begin{equation}\label{eq:rescaling}
\frac{\|u'\|_{L^1(0,1)}}
{\|u'''\|_{L^2(0,1)}}
=
\frac{1}{4\sqrt2}
\frac{\|Tu\|_{L^1(-1,1)}}
{\|(Tu)''\|_{L^2(-1,1)}},
\qquad
u\in\mathring W_2^3(0,1)\setminus\{0\}.
\end{equation}

\item[\rm(iv)]
Consequently,
\[
\lambda_3(3,1,2,1)
=
\frac{C_*}{4\sqrt2},
\]
and $u$ is an extremal function for $\lambda_3(3,1,2,1)$ if and only
if $Tu$ is an extremal function for \eqref{eq:Cstar}.
\end{enumerate}
\end{Lem}

\begin{proof}
(i) Let $V\in X$. By \eqref{eq:X},
\[
V(-1)=V(1)=V'(-1)=V'(1)=0
\]
and
\[
\int_{-1}^1 V(t)\,dt=0.
\]
Integration by parts gives
\[
\int_{-1}^1 V''(t)\,dt=0,
\qquad
\int_{-1}^1 tV''(t)\,dt=0,
\qquad
\int_{-1}^1 t^2V''(t)\,dt
=
2\int_{-1}^1V(t)\,dt
=
0.
\]
Hence
\[
V''\in\mathcal P_2^\perp,
\]
and therefore
\[
D^2(X)\subset\mathcal P_2^\perp.
\]

Conversely, let $f\in\mathcal P_2^\perp$ and define
\[
V(t)
=
\int_{-1}^t (t-s)f(s)\,ds,
\qquad -1\leq t\leq1.
\]
Then $V\in H^2(-1,1)$,
\[
V'(t)
=
\int_{-1}^t f(s)\,ds,
\qquad
V''(t)=f(t)
\quad\text{a.e. in }(-1,1),
\]
and
\[
V(-1)=V'(-1)=0.
\]
Since $f\perp\mathcal P_2$,
\[
V'(1)
=
\int_{-1}^1 f(s)\,ds
=
0,
\qquad
V(1)
=
\int_{-1}^1(1-s)f(s)\,ds
=
0.
\]
Thus
\[
V\in H_0^2(-1,1),
\qquad
D^2V=f.
\]
Moreover, by Fubini's theorem,
\[
\begin{aligned}
\int_{-1}^1V(t)\,dt
&=
\int_{-1}^1
\int_{-1}^t (t-s)f(s)\,ds\,dt\\
&=
\frac12
\int_{-1}^1(1-s)^2f(s)\,ds\\
&=
0,
\end{aligned}
\]
since $(1-s)^2\in\mathcal P_2$. Hence $V\in X$.

If $V\in X$ and $D^2V=0$, then $V$ is affine, and the boundary
conditions in \eqref{eq:X} imply $V\equiv0$. Thus $D^2$ is injective.
Consequently,
\[
D^2:X\longrightarrow\mathcal P_2^\perp
\]
is bijective.

Finally, let $f\in\mathcal P_2^\perp$. By \eqref{eq:rx} and
\eqref{eq:A},
\[
\begin{aligned}
(Af)(x)
&=
\int_{-1}^1 r_x(t)f(t)\,dt\\
&=
\int_{-1}^1 \phi_x(t)f(t)\,dt\\
&=
\int_{-1}^x (x-t)f(t)\,dt,
\end{aligned}
\]
since $\Pi_2\phi_x\in\mathcal P_2$ and
$f\perp\mathcal P_2$. Comparing with the preceding construction,
we obtain
\[
Af=(D^2)^{-1}f.
\]
Therefore,
\[
\left.A\right|_{\mathcal P_2^\perp}
=
(D^2)^{-1}.
\]

\noindent (ii) By \eqref{eq:Cstar} and (i),
\[
C_*
=
\sup_{f\in\mathcal P_2^\perp\setminus\{0\}}
\frac{\|Af\|_{L^1(-1,1)}}
{\|f\|_{L^2(-1,1)}}.
\]
Moreover, by \eqref{eq:rx} and \eqref{eq:A},
\[
Af=A(I-\Pi_2)f,
\qquad
f\in L^2(-1,1).
\]
Hence
\[
\|Af\|_{L^1(-1,1)}
\leq
C_*\|(I-\Pi_2)f\|_{L^2(-1,1)}
\leq
C_*\|f\|_{L^2(-1,1)},
\]
and therefore
\[
\|A\|_{2\to1}\leq C_*.
\]
Conversely, since $\mathcal P_2^\perp\subset L^2(-1,1)$,
\[
C_*
\leq
\sup_{f\in L^2(-1,1)\setminus\{0\}}
\frac{\|Af\|_{L^1(-1,1)}}
{\|f\|_{L^2(-1,1)}}
=
\|A\|_{2\to1}.
\]
Thus
\[
C_*=\|A\|_{2\to1}.
\]
Finally, the equality of the norms of a bounded operator and its Banach
adjoint gives
\[
\|A\|_{2\to1}
=
\|A^*\|_{\infty\to2}.
\]
This proves (ii).

\noindent (iii) Let $u\in\mathring W_2^3(0,1)$ and set
\[
V=Tu,
\qquad
V(t)=u'\left(\frac{t+1}{2}\right).
\]
Then $V\in H^2(-1,1)$ and
\[
V'(t)
=
\frac12u''\left(\frac{t+1}{2}\right).
\]
Since
\[
u'(0)=u'(1)=u''(0)=u''(1)=0,
\]
we have
\[
V(\pm1)=V'(\pm1)=0,
\]
and hence
\[
V\in H_0^2(-1,1).
\]
Moreover, by the change of variables $x=(t+1)/2$,
\[
\int_{-1}^1V(t)\,dt
=
2\int_0^1u'(x)\,dx
=
2\bigl(u(1)-u(0)\bigr)
=
0.
\]
Therefore,
\[
Tu=V\in X.
\]

Conversely, let $V\in X$ and define
\[
u(x)
=
\int_0^x V(2s-1)\,ds,
\qquad 0\leq x\leq1.
\]
Then
\[
u'(x)=V(2x-1),
\qquad
u''(x)=2V'(2x-1),
\qquad
u'''(x)=4V''(2x-1),
\]
and hence $u\in W_2^3(0,1)$. Since $V\in H_0^2(-1,1)$,
\[
u'(0)=u'(1)=u''(0)=u''(1)=0.
\]
Moreover,
\[
u(0)=0
\]
and, by the defining moment condition of $X$,
\[
u(1)
=
\int_0^1V(2s-1)\,ds
=
\frac12\int_{-1}^1V(t)\,dt
=
0.
\]
Hence
\[
u\in\mathring W_2^3(0,1),
\]
and
\[
(Tu)(t)
=
u'\left(\frac{t+1}{2}\right)
=
V(t).
\]
Thus $T$ is surjective.

Suppose that $Tu=0$ for some
$u\in\mathring W_2^3(0,1)$. Since $u'$ is continuous,
\[
u'(x)=0,
\qquad 0\leq x\leq1.
\]
Thus $u$ is constant, and since $u(0)=0$, we obtain $u=0$.
Hence
\[
\ker T=\{0\}.
\]
Since $T$ is linear, it is injective. Consequently, $T$ is a
bijection from $\mathring W_2^3(0,1)$ onto $X$.

Finally, for $u\in\mathring W_2^3(0,1)\setminus\{0\}$, a change of
variables gives
\[
\|Tu\|_{L^1(-1,1)}
=
2\|u'\|_{L^1(0,1)}
\]
and
\[
\|(Tu)''\|_{L^2(-1,1)}
=
\frac{1}{2\sqrt2}\,
\|u'''\|_{L^2(0,1)}.
\]
This yields \eqref{eq:rescaling} and completes the proof of (iii).

\noindent (iv) By the definition of the optimal constant in
\eqref{eq:lambda3},
\[
\lambda_3(3,1,2,1)
=
\sup_{u\in\mathring W_2^3(0,1)\setminus\{0\}}
\frac{\|u'\|_{L^1(0,1)}}
{\|u'''\|_{L^2(0,1)}}.
\]
Since $T$ is a bijection from
$\mathring W_2^3(0,1)$ onto $X$, taking the supremum in
\eqref{eq:rescaling} gives
\[
\lambda_3(3,1,2,1)
=
\frac{C_*}{4\sqrt2}.
\]
Moreover, by \eqref{eq:rescaling} and the bijectivity of $T$, the
left-hand quotient attains its supremum at $u$ if and only if the
right-hand quotient attains its supremum at $Tu$. Hence $u$ is an
extremal function for $\lambda_3(3,1,2,1)$ if and only if $Tu$ is an
extremal function for \eqref{eq:Cstar}. This proves (iv).
\end{proof}

\subsection{Reflection symmetry and decomposition of the adjoint}

We next exploit the reflection symmetry of the family
$\{r_x\}_{x\in(-1,1)}$ and use it to decompose the adjoint operator
$A^*$ into its odd and even components.

\begin{Lem}\label{lem:reflection}
The reflection operator $R_2$ defined by \eqref{eq:R2} has the
following properties.
\begin{enumerate}
\item[\rm(i)]
The operator $R_2$ is an orthogonal involution on $L^2(-1,1)$.
In particular,
\[
R_2^*=R_2^{-1}=R_2,
\qquad
R_2^2=I,
\]
and
\[
\|R_2f\|_{L^2(-1,1)}
=
\|f\|_{L^2(-1,1)},\qquad f\in L^2(-1,1).
\]

\item[\rm(ii)]
The subspaces $\mathcal P_2$ and $\mathcal P_2^\perp$ are invariant
under $R_2$, and
\[
\Pi_2R_2=R_2\Pi_2.
\]

\item[\rm(iii)]
The family $\{r_x\}_{x\in(-1,1)}$ defined by \eqref{eq:rx} satisfies
\[
r_{-x}=R_2r_x,
\qquad x\in(-1,1).
\]
\end{enumerate}
\end{Lem}

\begin{proof}
(i) By \eqref{eq:R2},
\[
R_2^2=I
\qquad\text{and}\qquad
\|R_2f\|_{L^2(-1,1)}
=
\|f\|_{L^2(-1,1)}
\]
for every $f\in L^2(-1,1)$. Hence $R_2$ is an orthogonal involution and
\[
R_2^{-1}=R_2.
\]
Since an orthogonal operator satisfies $R_2^*=R_2^{-1}$, we obtain
\[
R_2^*=R_2^{-1}=R_2.
\]

\noindent (ii) Let $g\in\mathcal P_2$. Then
\[
g(t)=a+bt+ct^2
\]
for some $a,b,c\in\mathbb R$. By \eqref{eq:R2},
\[
(R_2g)(t)
=
g(-t)
=
a-bt+ct^2
\in\mathcal P_2.
\]
Hence $\mathcal P_2$ is invariant under $R_2$.

Let now $h\in\mathcal P_2^\perp$ and $g\in\mathcal P_2$. By (i) and
the invariance of $\mathcal P_2$,
\[
\langle R_2h,g\rangle_{L^2(-1,1)}
=
\langle h,R_2g\rangle_{L^2(-1,1)}
=
0.
\]
Thus
\[
R_2h\in\mathcal P_2^\perp,
\]
so $\mathcal P_2^\perp$ is invariant under $R_2$.

Finally, for $f\in L^2(-1,1)$, write
\[
f=\Pi_2f+(I-\Pi_2)f.
\]
Applying $R_2$, we obtain
\[
R_2f
=
R_2\Pi_2f+R_2(I-\Pi_2)f.
\]
By the invariance of $\mathcal P_2$ and $\mathcal P_2^\perp$ under $R_2$,
\[
R_2\Pi_2f\in\mathcal P_2,
\qquad
R_2(I-\Pi_2)f\in\mathcal P_2^\perp.
\]
Hence, by uniqueness of the orthogonal decomposition,
\[
\Pi_2(R_2f)=R_2(\Pi_2f),
\]
and therefore
\[
\Pi_2R_2=R_2\Pi_2.
\]

\noindent (iii) Let $x\in(-1,1)$. By \eqref{eq:rx} and (ii),
\[
R_2r_x
=
R_2(I-\Pi_2)\phi_x
=
(I-\Pi_2)R_2\phi_x.
\]
By \eqref{eq:R2}, for almost every $t\in(-1,1)$,
\[
(R_2\phi_x)(t)
=
\phi_x(-t)
=
(x+t)_+.
\]
Moreover,
\[
(R_2\phi_x)(t)-\phi_{-x}(t)
=
(x+t)_+-(-x-t)_+
=
x+t.
\]
Hence
\[
R_2\phi_x-\phi_{-x}\in\mathcal P_2.
\]
Applying $I-\Pi_2$, we obtain
\[
(I-\Pi_2)R_2\phi_x
=
(I-\Pi_2)\phi_{-x}.
\]
Therefore, by \eqref{eq:rx},
\[
R_2r_x=r_{-x}.
\]
This proves (iii).
\end{proof}

We next analyze the adjoint operator $A^*$ defined by \eqref{eq:A-star}.

For $g\in L^\infty(-1,1)$ and almost every $x\in(0,1)$, set
\begin{equation}\label{eq:ag-bg}
a_g(x)
=
\frac{g(x)-g(-x)}{2},
\qquad
b_g(x)
=
\frac{g(x)+g(-x)}{2}.
\end{equation}
Then $a_g,b_g\in L^\infty(0,1)$.

\begin{Lem}\label{lem:adjoint-decomposition}
The following assertions hold.
\begin{enumerate}

\item[\rm(i)]
The adjoint operator $A^*$ intertwines the reflection operators
$R_\infty$ and $R_2$ defined by \eqref{eq:Rinfty} and \eqref{eq:R2},
respectively:
\[
A^*R_\infty
=
R_2A^*.
\]

\item[\rm(ii)]
For every $g\in L^\infty(-1,1)$,
\[
\|A^*g\|_{L^2(-1,1)}^2
=
\left\|
\int_0^1
a_g(x)\bigl(r_x-r_{-x}\bigr)\,dx
\right\|_{L^2(-1,1)}^2
+
\left\|
\int_0^1
b_g(x)\bigl(r_x+r_{-x}\bigr)\,dx
\right\|_{L^2(-1,1)}^2,
\]
where the integrals are understood as Bochner integrals in
$L^2(-1,1)$.
\end{enumerate}
\end{Lem}

\begin{proof}
(i) Let $g\in L^\infty(-1,1)$. By \eqref{eq:A-star} and
\eqref{eq:Rinfty},
\[
A^*R_\infty g
=
\int_{-1}^1 g(-x)r_x\,dx
=
\int_{-1}^1 g(x)r_{-x}\,dx,
\]
where the second equality follows from the change of variable
$x\mapsto -x$. By Lemma~\ref{lem:reflection}(iii),
\[
r_{-x}=R_2r_x,
\]
and therefore
\[
A^*R_\infty g
=
\int_{-1}^1 g(x)R_2r_x\,dx
=
R_2\int_{-1}^1 g(x)r_x\,dx
=
R_2A^*g.
\]
Hence
\[
A^*R_\infty=R_2A^*.
\]

\noindent (ii) Let $g\in L^\infty(-1,1)$. By \eqref{eq:ag-bg},
for almost every $x\in(0,1)$,
\[
g(x)=a_g(x)+b_g(x),
\qquad
g(-x)=b_g(x)-a_g(x).
\]
Splitting the integral in \eqref{eq:A-star} over $(-1,0)$ and
$(0,1)$, and using the change of variable $x\mapsto -x$ on
$(-1,0)$, we obtain, in $L^2(-1,1)$,
\begin{equation}\label{eq2:A-star}
A^*g
=
\int_0^1 a_g(x)\bigl(r_x-r_{-x}\bigr)\,dx
+
\int_0^1 b_g(x)\bigl(r_x+r_{-x}\bigr)\,dx.
\end{equation}

By Lemma~\ref{lem:reflection}(i) and (iii),
\[
R_2(r_x-r_{-x})
=
-(r_x-r_{-x}),
\qquad
R_2(r_x+r_{-x})
=
r_x+r_{-x}.
\]
Since $R_2$ is bounded and linear, it commutes with the above Bochner
integrals. Hence
\[
\int_0^1 a_g(x)\bigl(r_x-r_{-x}\bigr)\,dx
\]
is odd, whereas
\[
\int_0^1 b_g(x)\bigl(r_x+r_{-x}\bigr)\,dx
\]
is even. The two terms are therefore orthogonal in $L^2(-1,1)$.
Using \eqref{eq2:A-star}, we obtain
\[
\|A^*g\|_{L^2(-1,1)}^2
=
\left\|
\int_0^1 a_g(x)\bigl(r_x-r_{-x}\bigr)\,dx
\right\|_{L^2(-1,1)}^2
+
\left\|
\int_0^1 b_g(x)\bigl(r_x+r_{-x}\bigr)\,dx
\right\|_{L^2(-1,1)}^2.
\]
This proves (ii) and completes the proof.
\end{proof}

\subsection{Kernel identities and estimates}

For $x,y\in(-1,1)$, define the Gram kernel
\begin{equation}\label{eq:K}
K(x,y)
=
\langle r_x,r_y\rangle_{L^2(-1,1)}.
\end{equation}
For $x,y\in(0,1)$, set
\begin{equation}\label{eq:Kpm}
\begin{aligned}
K_-(x,y)
&=
K(x,y)-K(x,-y),\\
K_+(x,y)
&=
K(x,y)+K(x,-y).
\end{aligned}
\end{equation}
By \eqref{est:norm-rx} and the Cauchy--Schwarz inequality,
\[
K\in L^\infty((-1,1)^2),
\]
and hence
\[
K_-,K_+\in L^\infty((0,1)^2).
\]

We also define
\begin{equation}\label{eq:kappa-pm}
\begin{aligned}
\kappa_-(x)
&=
\int_0^1 K_-(x,y)\,dy,\\
\kappa_+(x)
&=
\frac{1}{\sqrt2}
\|r_x+r_{-x}\|_{L^2(-1,1)},
\end{aligned}
\qquad 0<x<1.
\end{equation}

\begin{Lem}\label{lem:kernel-properties}
The following assertions hold.
\begin{enumerate}

\item[\rm(i)]
For all $x,y\in(-1,1)$,
\[
K(-x,-y)=K(x,y),
\qquad
K(-x,y)=K(x,-y).
\]
Consequently,
\[
K_\pm(x,y)=K_\pm(y,x),
\qquad
0<x,y<1.
\]

\item[\rm(ii)]
For $0<x\leq y<1$,
\[
K_-(x,y)
=
\frac{x(1-y)^2}{12}
\left(
3y-x^2(y+2)
\right).
\]
In particular,
\[
K_-(x,y)>0,
\qquad
0<x,y<1.
\]

\item[\rm(iii)]
For $0<x<1$,
\[
\kappa_-(x)
=
\frac{x(1-x)^2(2x+1)}{48}.
\]

\item[\rm(iv)]
For $0<x<1$,
\[
\kappa_+(x)^2
=
K_+(x,x)
=
\frac{(1-x)^3}{192}P(x),
\]
where
\begin{equation}\label{eq:P}
P(x)
=
15x^5+45x^4+30x^3-30x^2+3x+1.
\end{equation}

\end{enumerate}
\end{Lem}

\begin{proof}
(i) Let $x,y\in(-1,1)$. By Lemma~\ref{lem:reflection}(i) and (iii),
\[
\begin{aligned}
K(-x,-y)
&=
\langle r_{-x},r_{-y}\rangle_{L^2(-1,1)}\\
&=
\langle R_2r_x,R_2r_y\rangle_{L^2(-1,1)}\\
&=
\langle r_x,r_y\rangle_{L^2(-1,1)}\\
&=
K(x,y),
\end{aligned}
\]
and
\[
\begin{aligned}
K(-x,y)
&=
\langle r_{-x},r_y\rangle_{L^2(-1,1)}\\
&=
\langle R_2r_x,r_y\rangle_{L^2(-1,1)}\\
&=
\langle r_x,R_2r_y\rangle_{L^2(-1,1)}\\
&=
\langle r_x,r_{-y}\rangle_{L^2(-1,1)}\\
&=
K(x,-y).
\end{aligned}
\]
Hence, by the symmetry of the inner product and \eqref{eq:Kpm}, for
$x,y\in(0,1)$,
\[
\begin{aligned}
K_\pm(y,x)
&=
K(y,x)\pm K(y,-x)\\
&=
K(x,y)\pm K(-x,y)\\
&=
K(x,y)\pm K(x,-y)\\
&=
K_\pm(x,y).
\end{aligned}
\]
This proves (i).

\noindent (ii) For $z\in(0,1)$, set
\[
d_z=r_z-r_{-z}.
\]
By Lemma~\ref{lem:reflection}(iii),
\begin{equation}\label{eq:odd-dz}
R_2d_z
=
R_2r_z-R_2r_{-z}
=
r_{-z}-r_z
=
-d_z,
\qquad 0<z<1.
\end{equation}
Using also Lemma~\ref{lem:reflection}(i), for all $x,y\in(0,1)$,
\[
\begin{aligned}
\langle r_{-x},d_y\rangle_{L^2(-1,1)}
&=
\langle R_2r_x,d_y\rangle_{L^2(-1,1)}\\
&=
\langle r_x,R_2d_y\rangle_{L^2(-1,1)}\\
&=
-\langle r_x,d_y\rangle_{L^2(-1,1)}.
\end{aligned}
\]
Hence
\[
\begin{aligned}
\langle d_x,d_y\rangle_{L^2(-1,1)}
&=
\langle r_x-r_{-x},d_y\rangle_{L^2(-1,1)}\\
&=
2\langle r_x,d_y\rangle_{L^2(-1,1)}.
\end{aligned}
\]
Consequently, by \eqref{eq:Kpm},
\begin{equation}\label{eq:Kminus-exp1}
\begin{aligned}
K_-(x,y)
&=
\langle r_x,r_y-r_{-y}\rangle_{L^2(-1,1)}\\
&=
\langle r_x,d_y\rangle_{L^2(-1,1)}\\
&=
\frac12
\langle d_x,d_y\rangle_{L^2(-1,1)},
\qquad 0<x,y<1.
\end{aligned}
\end{equation}

We next compute $d_z$. Recalling \eqref{eq:phi-x}, for $z\in(0,1)$ set
\[
\psi_z(t)
=
\phi_z(t)-\phi_{-z}(t)-z
=
(z-t)_+-(-z-t)_+-z,
\qquad -1<t<1.
\]
Since the constant function $t\mapsto z$ belongs to $\mathcal P_2$,
by \eqref{eq:rx},
\begin{equation}\label{eq:dz-psi}
d_z
=
(I-\Pi_2)\psi_z.
\end{equation}
The function $\psi_z$ is odd and satisfies
\begin{equation}\label{eq:psi-z-piecewise}
\psi_z(t)
=
\begin{cases}
z, & -1<t<-z,\\
-t, & -z\leq t\leq z,\\
-z, & z<t<1.
\end{cases}
\end{equation}
By Lemma~\ref{lem:reflection}(ii),
\[
R_2\Pi_2\psi_z
=
\Pi_2R_2\psi_z
=
-\Pi_2\psi_z.
\]
Hence $\Pi_2\psi_z$ is odd. Since $\Pi_2\psi_z\in\mathcal P_2$,
it follows that
\[
\Pi_2\psi_z\in\operatorname{span}\{t\}.
\]
Since
\[
\|t\|_{L^2(-1,1)}^2
=
\frac23
\]
and, by \eqref{eq:psi-z-piecewise},
\[
\begin{aligned}
\langle \psi_z,t\rangle_{L^2(-1,1)}
&=
-2\left(
\int_0^z t^2\,dt
+
z\int_z^1 t\,dt
\right)\\
&=
\frac{z^3-3z}{3},
\end{aligned}
\]
we obtain
\[
\Pi_2\psi_z(t)
=
\frac{z^3-3z}{2}\,t,
\qquad -1<t<1.
\]
Therefore, by \eqref{eq:dz-psi} and \eqref{eq:psi-z-piecewise}, setting
\begin{equation}\label{eq:betaz}
\beta_z
=
\frac{z(3-z^2)}{2},
\qquad 0<z<1,
\end{equation}
we obtain, for $0<t<1$,
\begin{equation}\label{eq:dz-piecewise}
d_z(t)
=
\begin{cases}
(\beta_z-1)t, & 0<t\leq z,\\
\beta_z t-z, & z<t<1.
\end{cases}
\end{equation}

Let $0<x\leq y<1$. By \eqref{eq:odd-dz} and
\eqref{eq:Kminus-exp1},
\[
K_-(x,y)
=
\int_0^1 d_x(t)d_y(t)\,dt.
\]
Using \eqref{eq:dz-piecewise},
\[
K_-(x,y)
=
\int_0^x
(\beta_x-1)(\beta_y-1)t^2\,dt
+
\int_x^y
(\beta_xt-x)(\beta_y-1)t\,dt
+
\int_y^1
(\beta_xt-x)(\beta_yt-y)\,dt.
\]
Using \eqref{eq:betaz} and simplifying, we obtain
\[
K_-(x,y)
=
\frac{x(1-y)^2}{12}
\left(
3y-x^2(y+2)
\right).
\]
Finally, since $0<x\leq y<1$,
\[
3y-x^2(y+2)
\geq
3y-y^2(y+2)
=
y(1-y)(y+3)
>
0.
\]
Thus
\[
K_-(x,y)>0,
\qquad
0<x\leq y<1.
\]
By (i), $K_-$ is symmetric, and hence
\[
K_-(x,y)>0,
\qquad
0<x,y<1.
\]
This proves (ii).

\noindent (iii) Let $0<x<1$. By \eqref{eq:kappa-pm} and the symmetry
of $K_-$ established in (i),
\[
\kappa_-(x)
=
\int_0^x K_-(y,x)\,dy
+
\int_x^1 K_-(x,y)\,dy.
\]
Using (ii), we obtain
\[
\kappa_-(x)
=
\frac{(1-x)^2}{12}
\int_0^x
y\left(3x-y^2(x+2)\right)\,dy
+
\frac{x}{12}
\int_x^1
(1-y)^2
\left(3y-x^2(y+2)\right)\,dy.
\]
A direct computation gives
\[
\kappa_-(x)
=
\frac{x(1-x)^2(2x+1)}{48}.
\]
This proves (iii).

\noindent (iv) Let $0<x<1$. By \eqref{eq:kappa-pm},
\[
\begin{aligned}
\kappa_+(x)^2
&=
\frac12
\|r_x+r_{-x}\|_{L^2(-1,1)}^2\\
&=
\frac12
\left(
\|r_x\|_{L^2(-1,1)}^2
+
\|r_{-x}\|_{L^2(-1,1)}^2
+
2\langle r_x,r_{-x}\rangle_{L^2(-1,1)}
\right).
\end{aligned}
\]
By Lemma~\ref{lem:reflection}(i) and (iii),
\[
\|r_{-x}\|_{L^2(-1,1)}
=
\|r_x\|_{L^2(-1,1)}.
\]
Therefore, by \eqref{eq:K},
\[
\begin{aligned}
\kappa_+(x)^2
&=
\|r_x\|_{L^2(-1,1)}^2
+
\langle r_x,r_{-x}\rangle_{L^2(-1,1)}\\
&=
K(x,x)+K(x,-x).
\end{aligned}
\]
Thus, by \eqref{eq:Kpm},
\begin{equation}\label{eq:kappa+-id1}
\kappa_+(x)^2
=
K_+(x,x).
\end{equation}

Set
\[
\eta_x(t)=(|t|-x)_+,
\qquad -1<t<1.
\]
By \eqref{eq:rx} and Lemma~\ref{lem:reflection}(ii)--(iii),
\[
\begin{aligned}
r_x+r_{-x}
&=
(I-\Pi_2)\phi_x+(I-\Pi_2)R_2\phi_x\\
&=
(I-\Pi_2)(\phi_x+R_2\phi_x).
\end{aligned}
\]
Since
\[
\phi_x(t)+(R_2\phi_x)(t)
=
2x+\eta_x(t),
\qquad -1<t<1,
\]
and the constant function $t\mapsto2x$ belongs to $\mathcal P_2$, we
obtain
\begin{equation}\label{eq:rx-plus-rminusx}
r_x+r_{-x}
=
(I-\Pi_2)\eta_x.
\end{equation}

The function $\eta_x$ is even. Since $\Pi_2$ commutes with $R_2$ by
Lemma~\ref{lem:reflection}(ii), $\Pi_2\eta_x$ is also even. Hence
\[
\Pi_2\eta_x
\in
\operatorname{span}\left\{1,t^2-\frac13\right\},
\]
where the displayed basis is orthogonal in $L^2(-1,1)$. Moreover,
\[
\|1\|_{L^2(-1,1)}^2=2,
\qquad
\left\|t^2-\frac13\right\|_{L^2(-1,1)}^2
=
\frac{8}{45},
\]
and a direct computation gives
\[
\|\eta_x\|_{L^2(-1,1)}^2
=
\frac{2(1-x)^3}{3},
\qquad
\langle\eta_x,1\rangle_{L^2(-1,1)}
=
(1-x)^2,
\]
and
\[
\left\langle
\eta_x,t^2-\frac13
\right\rangle_{L^2(-1,1)}
=
\frac{(1-x^2)^2}{6}.
\]
Therefore,
\[
\begin{aligned}
\|\Pi_2\eta_x\|_{L^2(-1,1)}^2
&=
\frac{\langle\eta_x,1\rangle_{L^2(-1,1)}^2}
{\|1\|_{L^2(-1,1)}^2}
+
\frac{
\left\langle\eta_x,t^2-\frac13\right\rangle_{L^2(-1,1)}^2}
{\left\|t^2-\frac13\right\|_{L^2(-1,1)}^2}\\
&=
\frac{(1-x)^4}{2}
+
\frac{5(1-x^2)^4}{32}.
\end{aligned}
\]
By \eqref{eq:kappa-pm} and \eqref{eq:rx-plus-rminusx},
\[
\begin{aligned}
2\kappa_+(x)^2
&=
\|(I-\Pi_2)\eta_x\|_{L^2(-1,1)}^2\\
&=
\|\eta_x\|_{L^2(-1,1)}^2
-
\|\Pi_2\eta_x\|_{L^2(-1,1)}^2\\
&=
\frac{2(1-x)^3}{3}
-
\frac{(1-x)^4}{2}
-
\frac{5(1-x^2)^4}{32}.
\end{aligned}
\]
Using \eqref{eq:P} and simplifying, we obtain
\[
\kappa_+(x)^2
=
\frac{(1-x)^3}{192}P(x).
\]
Together with \eqref{eq:kappa+-id1}, this proves (iv) and completes
the proof.
\end{proof}

\subsection{Estimates for the odd and even components}

We now use the kernel identities established above to estimate separately
the odd and even components in the decomposition \eqref{eq2:A-star} of
$A^*g$.

We first estimate the odd component.

\begin{Lem}\label{lem:odd-estimate}
For every $a\in L^\infty(0,1)$,
\[
\left\|
\int_0^1 a(x)\bigl(r_x-r_{-x}\bigr)\,dx
\right\|_{L^2(-1,1)}^2
\leq
2\int_0^1 \kappa_-(x)a(x)^2\,dx.
\]
Moreover, equality holds if and only if $a$ is constant almost
everywhere on $(0,1)$.
\end{Lem}

\begin{proof}
Let $a\in L^\infty(0,1)$. By \eqref{eq:Kminus-exp1} and Fubini's
theorem,
\[
\begin{aligned}
\left\|
\int_0^1 a(x)\bigl(r_x-r_{-x}\bigr)\,dx
\right\|_{L^2(-1,1)}^2
&=
\int_0^1\int_0^1
a(x)a(y)
\langle d_x,d_y\rangle_{L^2(-1,1)}
\,dx\,dy\\
&=
2\int_0^1\int_0^1
a(x)a(y)K_-(x,y)
\,dx\,dy.
\end{aligned}
\]
By Lemma~\ref{lem:kernel-properties}(ii),
\[
K_-(x,y)>0,
\qquad 0<x,y<1.
\]
Hence, using
\[
2a(x)a(y)\leq a(x)^2+a(y)^2,
\]
together with \eqref{eq:kappa-pm} and the symmetry of $K_-$ established
in Lemma~\ref{lem:kernel-properties}(i), we obtain
\[
\begin{aligned}
\left\|
\int_0^1 a(x)\bigl(r_x-r_{-x}\bigr)\,dx
\right\|_{L^2(-1,1)}^2
&\leq
\int_0^1\int_0^1
\bigl(a(x)^2+a(y)^2\bigr)K_-(x,y)
\,dx\,dy\\
&=
2\int_0^1
a(x)^2
\left(
\int_0^1 K_-(x,y)\,dy
\right)\,dx\\
&=
2\int_0^1 \kappa_-(x)a(x)^2\,dx.
\end{aligned}
\]

If equality holds, then
\[
\int_0^1\int_0^1
(a(x)-a(y))^2K_-(x,y)\,dx\,dy
=
0.
\]
Since $K_-(x,y)>0$ on $(0,1)^2$, it follows that
\[
a(x)=a(y)
\]
for almost every $(x,y)\in(0,1)^2$. By Fubini's theorem, $a$ is
constant almost everywhere on $(0,1)$.

Conversely, if $a$ is constant almost everywhere, then
\[
2a(x)a(y)=a(x)^2+a(y)^2
\]
for almost every $(x,y)\in(0,1)^2$, and equality holds.
\end{proof}

\begin{Lem}\label{lem:even-estimate}
Let $h\in L^\infty(0,1)$ satisfy
\[
0\leq h(x)\leq1
\qquad\text{for almost every }x\in(0,1),
\]
and
\[
\|h\|_{L^\infty(0,1)}>0.
\]
Then
\begin{equation}\label{est-L2.7}
\left(
\int_0^1 \kappa_+(x)h(x)\,dx
\right)^2
<
\int_0^1 \kappa_-(x)h(x)\,dx.
\end{equation}
\end{Lem}

\begin{proof}
By Lemma~\ref{lem:kernel-properties}(iii),
\[
\kappa_-(x)>0,
\qquad 0<x<1.
\]
Define the finite measure $\mu$ on $(0,1)$ by
\[
d\mu(x)=\kappa_-(x)\,dx,
\]
and set
\[
\rho(x)
=
\frac{\kappa_+(x)}{\kappa_-(x)},
\qquad 0<x<1.
\]
By \eqref{eq:kappa-pm},
\[
\rho(x)\geq0,
\qquad 0<x<1.
\]
Moreover,
\begin{equation}\label{eq:q-mu-identities}
\int_0^1 \kappa_+(x)h(x)\,dx
=
\int_0^1 \rho(x)h(x)\,d\mu(x),
\qquad
\int_0^1 \kappa_-(x)h(x)\,dx
=
\int_0^1 h(x)\,d\mu(x).
\end{equation}

We first establish a weak-$L^2(\mu)$ estimate for $\rho$. By
Lemma~\ref{lem:kernel-properties}(iii) and (iv),
\begin{equation}\label{eq:q-square}
\rho(x)^2
=
\frac{12P(x)}
{x^2(1-x)(2x+1)^2},
\qquad 0<x<1.
\end{equation}
Set
\[
W(x)
=
\mu((0,x))
=
\int_0^x \kappa_-(s)\,ds,
\qquad 0\leq x\leq 1.
\]
Using Lemma~\ref{lem:kernel-properties}(iii), a direct integration gives
\begin{equation}\label{eq:W}
W(x)
=
\frac{x^2(8x^3-15x^2+10)}{960},
\qquad 0\leq x\leq 1,
\end{equation}
and, in particular,
\begin{equation}\label{eq:W1}
W(1)=\frac1{320}.
\end{equation}

We claim that
\begin{equation}\label{eq:q-tail-left}
\rho(x)^2W(x)
<
\frac18,
\qquad
0<x\leq\frac12,
\end{equation}
and
\begin{equation}\label{eq:q-tail-right}
\rho(x)^2\bigl(W(1)-W(x)\bigr)
<
\frac18,
\qquad
\frac12\leq x<1.
\end{equation}

Indeed, by \eqref{eq:P}, \eqref{eq:q-square}, and \eqref{eq:W},
\begin{equation}\label{eq:diff}
\frac18-\rho(x)^2W(x)
=
\frac{3x^2R_-(x)}
{80(x-1)(2x+1)^2},
\qquad
0<x\leq\frac12,
\end{equation}
where
\[
R_-(x)
=
40x^6+45x^5-145x^4-180x^3+308x^2+101x-105.
\]
For integers $0\leq i\leq j$, let
\[
B_{i,j}(s)
=
\binom{j}{i}s^i(1-s)^{j-i},
\qquad
0\leq s\leq1.
\]
A direct expansion in the Bernstein basis gives
\[
-R_-\left(\frac{s}{2}\right)
=
\sum_{i=0}^6 c_iB_{i,6}(s),
\qquad
0\leq s\leq1,
\]
where
\[
(c_0,\ldots,c_6)
=
\left(
105,\frac{1159}{12},\frac{2491}{30},\frac{2619}{40},
\frac{3651}{80},\frac{4919}{192},\frac{225}{32}
\right).
\]
All these coefficients are  positive, and hence
\[
R_-(x)<0,
\qquad
0\leq x\leq\frac12.
\]
Together with \eqref{eq:diff}, this yields \eqref{eq:q-tail-left}.

Similarly, using \eqref{eq:q-square}, \eqref{eq:W}, and
\eqref{eq:W1},
\begin{equation}\label{eq:diff2}
\frac18
-
\rho(x)^2\bigl(W(1)-W(x)\bigr)
=
-\frac{R_+(x)}
{80x^2(2x+1)^2},
\qquad
\frac12\leq x<1,
\end{equation}
where
\[
R_+(x)
=
120x^9+255x^8-180x^7-720x^6+204x^5
+382x^4-68x^3-98x^2+12x+3.
\]
A direct expansion in the Bernstein basis gives
\[
-R_+\left(\frac{1+s}{2}\right)
=
\sum_{i=0}^9 d_iB_{i,9}(s),
\qquad
0\leq s\leq1,
\]
where
\[
(d_0,\ldots,d_9)
=
\Bigg(
\frac{1325}{256},
\frac{4691}{768},
\frac{8269}{1152},
\frac{6019}{672},
\frac{1577}{126},
\frac{2453}{126},
\frac{1775}{56},
\frac{3625}{72},
\frac{220}{3},
90
\Bigg).
\]
Again all coefficients are  positive. Thus
\[
R_+(x)<0,
\qquad
\frac12\leq x\leq1.
\]
Together with \eqref{eq:diff2}, this yields
\eqref{eq:q-tail-right}.

We claim that
\begin{equation}\label{eq:q-weak-L2}
\mu\bigl(\{x\in(0,1):\rho(x)>\lambda\}\bigr)
<
\frac{1}{4\lambda^2},
\qquad
\lambda>0.
\end{equation}
For $\lambda>0$, set
\[
c_\lambda=\frac{1}{8\lambda^2},
\]
and let
\[
E_\lambda^-
=
\left\{
0<x\leq\frac12:\rho(x)>\lambda
\right\}.
\]
By \eqref{eq:q-tail-left},
\[
W(x)<c_\lambda,
\qquad
x\in E_\lambda^-,
\]
and hence
\begin{equation}\label{incl1}
E_\lambda^-
\subset
\left\{
0<x\leq\frac12:
W(x)<c_\lambda
\right\}.
\end{equation}

If
\[
c_\lambda>W\left(\frac12\right),
\]
then, since $W$ is strictly increasing,
\[
\left\{
0<x\leq\frac12:
W(x)<c_\lambda
\right\}
=
\left(0,\frac12\right],
\]
and therefore
\[
\mu(E_\lambda^-)
\leq
W\left(\frac12\right)
<
c_\lambda.
\]

Suppose now that
\[
0<c_\lambda\leq W\left(\frac12\right).
\]
Since $W$ is continuous and strictly increasing, with $W(0)=0$, there
exists a unique $\xi\in(0,\frac12]$ such that
\[
W(\xi)=c_\lambda.
\]
Thus, by \eqref{incl1},
\[
E_\lambda^-\subset(0,\xi).
\]
Moreover, \eqref{eq:q-tail-left} gives
\[
\rho(\xi)^2W(\xi)<\frac18,
\]
and hence
\[
\rho(\xi)<\lambda.
\]
Since $\rho$ is continuous on $(0,1)$, there exists
$0<\delta<\xi$ such that
\[
\rho(x)<\lambda,
\qquad
\xi-\delta\leq x\leq\xi.
\]
Consequently,
\[
E_\lambda^-\subset(0,\xi-\delta),
\]
and therefore
\[
\mu(E_\lambda^-)
\leq
W(\xi-\delta)
<
W(\xi)
=
c_\lambda.
\]
Thus, in either case,
\begin{equation}\label{eq:weak-left}
\mu(E_\lambda^-)
<
c_\lambda.
\end{equation}

Likewise, let
\[
E_\lambda^+
=
\left\{
\frac12<x<1:\rho(x)>\lambda
\right\}.
\]
By \eqref{eq:q-tail-right},
\[
W(1)-W(x)
<
c_\lambda,
\qquad x\in E_\lambda^+,
\]
and hence
\[
E_\lambda^+
\subset
\left\{
\frac12<x<1:
W(1)-W(x)<c_\lambda
\right\}.
\]
Since $x\mapsto W(1)-W(x)$ is continuous and strictly decreasing on
$[\frac12,1]$, there exists $\xi\in[\frac12,1)$ such that
\[
\left\{
\frac12<x<1:
W(1)-W(x)<c_\lambda
\right\}
=
(\xi,1),
\]
with
\[
W(1)-W(\xi)\leq c_\lambda.
\]
Therefore,
\begin{equation}\label{eq:weak-right}
\begin{aligned}
\mu(E_\lambda^+)
&\leq
\mu((\xi,1))\\
&=
W(1)-W(\xi)\\
&\leq
c_\lambda.
\end{aligned}
\end{equation}
Since
\[
\{x\in(0,1):\rho(x)>\lambda\}
=
E_\lambda^-\cup E_\lambda^+,
\]
and the union is disjoint, \eqref{eq:weak-left} and
\eqref{eq:weak-right} give
\[
\begin{aligned}
\mu\bigl(\{x\in(0,1):\rho(x)>\lambda\}\bigr)
&=
\mu(E_\lambda^-)+\mu(E_\lambda^+)\\
&<
2c_\lambda\\
&=
\frac{1}{4\lambda^2}.
\end{aligned}
\]
This proves \eqref{eq:q-weak-L2}.

Now let $E\subset(0,1)$ be $\mu$-measurable with
\[
m=\mu(E)>0.
\]
Set
\[
\lambda_0=\frac{1}{2\sqrt{m}}.
\]
By \eqref{eq:q-weak-L2},
\[
\mu\bigl(E\cap\{\rho>\lambda\}\bigr)
<
\frac{1}{4\lambda^2},
\qquad
\lambda>\lambda_0,
\]
and hence
\[
\int_{\lambda_0}^{+\infty}
\mu\bigl(E\cap\{\rho>\lambda\}\bigr)\,d\lambda
<
\frac{1}{4\lambda_0}.
\]
On the other hand,
\[
\mu\bigl(E\cap\{\rho>\lambda\}\bigr)
\leq m,
\qquad
0<\lambda\leq\lambda_0,
\]
so
\[
\int_0^{\lambda_0}
\mu\bigl(E\cap\{\rho>\lambda\}\bigr)\,d\lambda
\leq
m\lambda_0.
\]
Therefore, by the layer-cake representation,
\[
\begin{aligned}
\int_E\rho\,d\mu
&=
\int_0^{+\infty}
\mu\bigl(E\cap\{\rho>\lambda\}\bigr)\,d\lambda\\
&<
m\lambda_0+\frac{1}{4\lambda_0}\\
&=
\sqrt{m}.
\end{aligned}
\]
Thus
\begin{equation}\label{eq:q-set-estimate}
\left(
\int_E \rho\,d\mu
\right)^2
<
\mu(E)
\end{equation}
for every $\mu$-measurable $E\subset(0,1)$ with $\mu(E)>0$.

Finally, for $0<t<1$, set
\[
E_t
=
\{x\in(0,1):h(x)>t\}.
\]
Since $\|h\|_{L^\infty(0,1)}>0$, there exists
\[
0<\varepsilon<\|h\|_{L^\infty(0,1)}
\]
such that
\[
|E_\varepsilon|>0.
\]
Since $\kappa_->0$ on $(0,1)$,
\[
\mu(E_\varepsilon)>0.
\]
Moreover,
\[
E_\varepsilon\subset E_t,
\qquad
0<t<\varepsilon,
\]
and hence
\[
\mu(E_t)>0,
\qquad
0<t<\varepsilon.
\]
By \eqref{eq:q-set-estimate},
\[
\int_{E_t}\rho\,d\mu
<
\mu(E_t)^{1/2},
\qquad
0<t<\varepsilon,
\]
whereas, for every $0<t<1$,
\[
\int_{E_t}\rho\,d\mu
\leq
\mu(E_t)^{1/2}.
\]
Consequently, by the layer-cake representation and the
Cauchy--Schwarz inequality,
\[
\begin{aligned}
\int_0^1 \rho(x)h(x)\,d\mu(x)
&=
\int_0^1
\left(
\int_{E_t}\rho(x)\,d\mu(x)
\right)dt\\
&<
\int_0^1\mu(E_t)^{1/2}\,dt\\
&\leq
\left(
\int_0^1\mu(E_t)\,dt
\right)^{1/2}\\
&=
\left(
\int_0^1 h(x)\,d\mu(x)
\right)^{1/2}.
\end{aligned}
\]
Hence
\[
\left(
\int_0^1 \rho(x)h(x)\,d\mu(x)
\right)^2
<
\int_0^1 h(x)\,d\mu(x).
\]
Together with \eqref{eq:q-mu-identities}, this yields
\eqref{est-L2.7}.
\end{proof}

\begin{Cor}\label{cor:even-component}
Let $b\in L^\infty(0,1)$ satisfy
\[
\|b\|_{L^\infty(0,1)}>0.
\]
Then
\[
\left\|
\int_0^1 b(x)\bigl(r_x+r_{-x}\bigr)\,dx
\right\|_{L^2(-1,1)}^2
<
2\|b\|_{L^\infty(0,1)}
\int_0^1 \kappa_-(x)|b(x)|\,dx.
\]
\end{Cor}

\begin{proof}
Set
\[
M=\|b\|_{L^\infty(0,1)},
\qquad
h(x)=\frac{|b(x)|}{M}.
\]
Then
\[
0\leq h(x)\leq1
\qquad\text{for almost every }x\in(0,1),
\]
and
\[
\|h\|_{L^\infty(0,1)}=1.
\]
By Lemma~\ref{lem:even-estimate},
\[
\left(
\int_0^1 \kappa_+(x)|b(x)|\,dx
\right)^2
<
M
\int_0^1 \kappa_-(x)|b(x)|\,dx.
\]
On the other hand, by the triangle inequality for Bochner integrals and
\eqref{eq:kappa-pm},
\[
\begin{aligned}
\left\|
\int_0^1 b(x)\bigl(r_x+r_{-x}\bigr)\,dx
\right\|_{L^2(-1,1)}
&\leq
\int_0^1
|b(x)|\,\|r_x+r_{-x}\|_{L^2(-1,1)}\,dx\\
&=
\sqrt2
\int_0^1 \kappa_+(x)|b(x)|\,dx.
\end{aligned}
\]
Consequently,
\[
\begin{aligned}
\left\|
\int_0^1 b(x)\bigl(r_x+r_{-x}\bigr)\,dx
\right\|_{L^2(-1,1)}^2
&\leq
2
\left(
\int_0^1 \kappa_+(x)|b(x)|\,dx
\right)^2\\
&<
2M
\int_0^1 \kappa_-(x)|b(x)|\,dx.
\end{aligned}
\]
This proves the corollary.
\end{proof}

\section{Proof of Theorem~\ref{thm:main}}\label{sec3}

\begin{proof}
\noindent\textbf{Step 1: The sharp constant.}

Let $g\in L^\infty(-1,1)$ satisfy
\[
\|g\|_{L^\infty(-1,1)}\leq1.
\]
By \eqref{eq:ag-bg}, for almost every $x\in(0,1)$,
\begin{equation}\label{eq:est-mod-ag-bg}
|a_g(x)|+|b_g(x)|
=
\max\{|g(x)|,|g(-x)|\}
\leq1.
\end{equation}

By Lemma~\ref{lem:adjoint-decomposition}(ii), Lemma~\ref{lem:odd-estimate}, and Corollary~\ref{cor:even-component}, with the case $b_g=0$ being trivial, 
\begin{align}
&\|A^*g\|_{L^2(-1,1)}^2 =\left\|
\int_0^1
a_g(x)\bigl(r_x-r_{-x}\bigr)\,dx
\right\|_{L^2(-1,1)}^2
+
\left\|
\int_0^1
b_g(x)\bigl(r_x+r_{-x}\bigr)\,dx
\right\|_{L^2(-1,1)}^2, \label{eq:conseq1}\\
&\left\|
\int_0^1 a_g(x)\bigl(r_x-r_{-x}\bigr)\,dx
\right\|_{L^2(-1,1)}^2
\leq
2\int_0^1 \kappa_-(x)a_g(x)^2\,dx,\label{eq:conseq2}\\
&\left\|
\int_0^1 b_g(x)\bigl(r_x+r_{-x}\bigr)\,dx
\right\|_{L^2(-1,1)}^2
\leq 
2\|b_g\|_{L^\infty(0,1)}
\int_0^1 \kappa_-(x)|b_g(x)|\,dx. \label{eq:conseq3}
\end{align}
Together with \eqref{eq:est-mod-ag-bg}, this yields
\begin{equation}\label{eq:Astar-chain}
\begin{aligned}
\|A^*g\|_{L^2(-1,1)}^2 &\leq
2\int_0^1 \kappa_-(x)a_g(x)^2\,dx
+
2\|b_g\|_{L^\infty(0,1)}
\int_0^1 \kappa_-(x)|b_g(x)|\,dx\\
&\leq
2\int_0^1
\kappa_-(x)
\left(
a_g(x)^2+|b_g(x)|
\right)\,dx\\
&\leq
2\int_0^1
\kappa_-(x)
\left(
|a_g(x)|+|b_g(x)|
\right)\,dx\\
&\leq
2\int_0^1 \kappa_-(x)\,dx.
\end{aligned}
\end{equation}
Using Lemma~\ref{lem:kernel-properties}(iii), a direct integration gives
\begin{equation}\label{eq:value-int-k-minus}
\int_0^1 \kappa_-(x)\,dx
=
\frac1{320}.
\end{equation}
Therefore,
\[
\|A^*g\|_{L^2(-1,1)}^2
\leq
\frac1{160},
\]
and hence, by \eqref{eq:norms},
\begin{equation}\label{eq:Astar-upper}
\|A^*\|_{\infty\to2}
\leq
\frac{1}{4\sqrt{10}}.
\end{equation}

Now let
\begin{equation}\label{eq:def-g0}
g_0(t)=\operatorname{sgn}(t),
\qquad -1<t<1.
\end{equation}
Then
\[
\|g_0\|_{L^\infty(-1,1)}=1,
\qquad
a_{g_0}(x)=1,
\qquad
b_{g_0}(x)=0,
\qquad
0<x<1.
\]
Since $a_{g_0}$ is constant, the equality case in
Lemma~\ref{lem:odd-estimate}, together with
Lemma~\ref{lem:adjoint-decomposition}(ii) and
\eqref{eq:value-int-k-minus}, gives
\[
\|A^*g_0\|_{L^2(-1,1)}^2
=
2\int_0^1\kappa_-(x)\,dx
=
\frac1{160}.
\]
Thus equality is attained in \eqref{eq:Astar-upper}, and
\[
\|A^*\|_{\infty\to2}
=
\frac{1}{4\sqrt{10}}.
\]

By Lemma~\ref{lem:reduction}(ii),
\[
C_*
=
\frac{1}{4\sqrt{10}},
\]
and Lemma~\ref{lem:reduction}(iv) therefore yields
\[
\lambda_3(3,1,2,1)
=
\frac{C_*}{4\sqrt2}
=
\frac{1}{32\sqrt5}.
\]
This proves \eqref{eq:sharp-cst}.

\medskip
\noindent\textbf{Step 2: Extremal functions for
$\lambda_3(3,1,2,1)$.}

We first determine all functions in the closed unit ball of
$L^\infty(-1,1)$ at which the norm $\|A^*\|_{\infty\to2}$ is attained.

Let $g\in L^\infty(-1,1)$ satisfy
\[
\|g\|_{L^\infty(-1,1)}\leq1
\]
and
\begin{equation}\label{eq:value-norm-Astarg}
\|A^*g\|_{L^2(-1,1)}
=
\|A^*\|_{\infty\to2}
=
\frac{1}{4\sqrt{10}}.
\end{equation}
Then equality holds throughout \eqref{eq:Astar-chain}. In view of
\eqref{eq:conseq1}, \eqref{eq:conseq2}, and \eqref{eq:conseq3}, this yields
\begin{equation}\label{eq:conseq1-equality}
\left\|
\int_0^1 a_g(x)\bigl(r_x-r_{-x}\bigr)\,dx
\right\|_{L^2(-1,1)}^2
=
2\int_0^1 \kappa_-(x)a_g(x)^2\,dx
\end{equation}
and
\begin{equation}\label{eq:conseq2-equality}
\left\|
\int_0^1 b_g(x)\bigl(r_x+r_{-x}\bigr)\,dx
\right\|_{L^2(-1,1)}^2
=
2\|b_g\|_{L^\infty(0,1)}
\int_0^1 \kappa_-(x)|b_g(x)|\,dx.
\end{equation}
By \eqref{eq:conseq2-equality} and
Corollary~\ref{cor:even-component},
\begin{equation}\label{b-g-zero}
b_g(x)=0
\qquad\text{for almost every }x\in(0,1).
\end{equation}
Hence, by Lemma~\ref{lem:adjoint-decomposition}(ii),
\begin{equation}\label{eq:norm-simplified}
\|A^*g\|_{L^2(-1,1)}^2
=
\left\|
\int_0^1
a_g(x)\bigl(r_x-r_{-x}\bigr)\,dx
\right\|_{L^2(-1,1)}^2.
\end{equation}
Moreover, by \eqref{eq:conseq1-equality} and
Lemma~\ref{lem:odd-estimate}, $a_g$ is constant almost everywhere on
$(0,1)$. Thus, for some $c\in\mathbb R$,
\begin{equation}\label{eq:ag-cosntant}
a_g(x)=c
\qquad\text{for almost every }x\in(0,1).
\end{equation}
Together with \eqref{eq:value-int-k-minus},
\eqref{eq:value-norm-Astarg}, \eqref{eq:conseq1-equality}, and
\eqref{eq:norm-simplified}, this gives
\[
\|A^*g\|_{L^2(-1,1)}^2
=
2c^2\int_0^1\kappa_-(x)\,dx
=
\frac{c^2}{160}
=
\frac1{160}.
\]
It follows that
\begin{equation}\label{eq:value-c}
c=\pm1.
\end{equation}
By \eqref{eq:ag-bg}, \eqref{b-g-zero},
\eqref{eq:ag-cosntant}, and \eqref{eq:value-c}, we conclude that
\[
g(t)=\pm g_0(t)
\qquad\text{for almost every }t\in(-1,1),
\]
where $g_0$ is defined by \eqref{eq:def-g0}.

Combining this characterization with Step~1, we conclude that the norm
of $A^*$ on the closed unit ball of $L^\infty(-1,1)$ is attained
exactly at $\pm g_0$. In particular, by Lemma~\ref{lem:reduction}(ii),
\begin{equation}\label{norm-A-starg-zero-L2}
\|A^*\|_{\infty\to2}
=
\|A^*g_0\|_{L^2(-1,1)}
=
C_*.
\end{equation}

We next compute $A^*g_0$ explicitly. By \eqref{eq:rx},
\eqref{eq:A-star}, and \eqref{eq:def-g0},
\begin{equation}\label{eq:Astar-g-0-proj}
A^*g_0
=
(I-\Pi_2)F,
\end{equation}
where
\[
F(t)
=
\int_{-1}^1
\operatorname{sgn}(x)(x-t)_+\,dx,
\qquad -1<t<1.
\]
A direct computation gives
\begin{equation}\label{eq:def-functionF}
F(t)
=
\frac12-t+\frac12t|t|,
\qquad -1<t<1.
\end{equation}
Since the even part of $F$ is the constant $1/2$, while the odd
subspace of $\mathcal P_2$ is spanned by $t$, we have
\[
\Pi_2F(t)
=
\frac12
+
\frac{\langle -t+\frac12t|t|,t\rangle_{L^2(-1,1)}}
{\|t\|_{L^2(-1,1)}^2}\,t.
\]
Moreover,
\[
\|t\|_{L^2(-1,1)}^2=\frac23,
\qquad
\left\langle \frac12t|t|,t\right\rangle_{L^2(-1,1)}
=
\frac12\int_{-1}^1 t^2|t|\,dt
=
\frac14.
\]
Hence
\[
\Pi_2F(t)
=
\frac12-\frac58t,
\qquad -1<t<1.
\]
Together with \eqref{eq:Astar-g-0-proj} and
\eqref{eq:def-functionF}, this yields
\begin{equation}\label{eq:f0}
f_0(t)
:=
(A^*g_0)(t)
=
\frac12t|t|-\frac38t
=
\frac{t(4|t|-3)}{8},
\qquad -1<t<1.
\end{equation}

Define
\begin{equation}\label{eq:V0}
V_0(t)
=
\frac{t(1-|t|)^2(1+2|t|)}{48},
\qquad -1\leq t\leq1.
\end{equation}
The function $V_0$ is odd,
\[
V_0(\pm1)=V_0'(\pm1)=0,
\]
and a direct differentiation gives
\begin{equation}\label{derv-2-V0}
V_0''=f_0,
\end{equation}
where $f_0$ is defined by \eqref{eq:f0}. Hence
$V_0\in H_0^2(-1,1)$, and its oddness yields
\[
\int_{-1}^1V_0(t)\,dt=0.
\]
Therefore,
\[
V_0\in X.
\]
By Lemma~\ref{lem:reduction}(i),
\begin{equation}\label{eq:appartenance}
f_0\in\mathcal P_2^\perp
\qquad\text{and}\qquad
V_0=Af_0.
\end{equation}
Moreover, by \eqref{eq:def-g0} and \eqref{eq:V0},
\begin{equation}\label{eq:sign-V0}
\operatorname{sgn}(V_0(t))
=
\operatorname{sgn}(t)
=
g_0(t),
\qquad -1<t<1.
\end{equation}
Together with \eqref{eq:f0} and \eqref{eq:appartenance}, this yields
\begin{equation}\label{normV0L1}
\begin{aligned}
\|V_0\|_{L^1(-1,1)}
&=
\int_{-1}^1g_0(t)V_0(t)\,dt\\
&=
\int_{-1}^1g_0(t)(Af_0)(t)\,dt\\
&=
\langle A^*g_0,f_0\rangle_{L^2(-1,1)}\\
&=
\|f_0\|_{L^2(-1,1)}^2.
\end{aligned}
\end{equation}
Hence, by \eqref{norm-A-starg-zero-L2}, \eqref{eq:f0},
\eqref{derv-2-V0}, and \eqref{normV0L1},
\[
\frac{\|V_0\|_{L^1(-1,1)}}
{\|V_0''\|_{L^2(-1,1)}}
=
\|f_0\|_{L^2(-1,1)}
=
\|A^*g_0\|_{L^2(-1,1)}
=
C_*.
\]
Thus $V_0$ is an extremal function for \eqref{eq:Cstar}.

We now prove that the extremal functions for \eqref{eq:Cstar} are unique
up to multiplication by a nonzero constant.

Let $V\in X\setminus\{0\}$ be an arbitrary extremal function for
\eqref{eq:Cstar}, and set
\[
f=V''.
\]
By Lemma~\ref{lem:reduction}(i),
\[
f\in\mathcal P_2^\perp,
\qquad
V=Af.
\]
Set
\[
g(t)=\operatorname{sgn}(V(t)),
\qquad -1<t<1.
\]
Then
\[
\|g\|_{L^\infty(-1,1)}=1
\]
and
\[
\|V\|_{L^1(-1,1)}
=
\int_{-1}^1g(t)V(t)\,dt.
\]
Since $V$ is extremal, by \eqref{eq:Cstar}, \eqref{eq:norms}, and
Lemma~\ref{lem:reduction}(ii),
\[
\begin{aligned}
C_*\|f\|_{L^2(-1,1)}
&=
\|V\|_{L^1(-1,1)}\\
&=
\int_{-1}^1g(t)(Af)(t)\,dt\\
&=
\langle A^*g,f\rangle_{L^2(-1,1)}\\
&\leq
\|A^*g\|_{L^2(-1,1)}
\|f\|_{L^2(-1,1)}\\
&\leq
\|A^*\|_{\infty\to2}
\|g\|_{L^\infty(-1,1)}
\|f\|_{L^2(-1,1)}\\
&=
C_*\|f\|_{L^2(-1,1)}.
\end{aligned}
\]
Since $f\neq0$, all the inequalities above are equalities. Consequently,
\begin{equation}\label{eq:equality-Cauchy-Schwarz}
\langle A^*g,f\rangle_{L^2(-1,1)}
=
\|A^*g\|_{L^2(-1,1)}
\|f\|_{L^2(-1,1)}
\end{equation}
and
\begin{equation}\label{eq:norm-Astarg-Cstar}
\|A^*g\|_{L^2(-1,1)}
=
C_*.
\end{equation}
By \eqref{eq:norm-Astarg-Cstar} and the characterization above,
\[
g=\pm g_0,
\]
and therefore, by \eqref{eq:f0},
\[
A^*g=\pm f_0.
\]
Moreover, by \eqref{eq:equality-Cauchy-Schwarz}, $f$ and $A^*g$ are
linearly dependent. Hence
\[
f=cf_0
\]
for some $c\neq0$. Since $V=Af$ and, by \eqref{eq:appartenance},
$V_0=Af_0$, the linearity of $A$ gives
\[
V=cV_0.
\]
Thus the extremal functions for \eqref{eq:Cstar} are precisely the
nonzero scalar multiples of $V_0$.

We now return to the original interval $(0,1)$. By
Lemma~\ref{lem:reduction}(iii), the inverse image of $V_0$ under $T$
is given by
\[
x\longmapsto
\int_0^x V_0(2s-1)\,ds.
\]
For convenience, we choose the normalization
\begin{equation}\label{eq:U-star}
U_*(x)
=
-120\int_0^x V_0(2s-1)\,ds,
\qquad 0\leq x\leq1.
\end{equation}
Then
\[
(TU_*)(t)
=
U_*'\left(\frac{t+1}{2}\right)
=
-120V_0(t),
\qquad -1\leq t\leq 1.
\]
Using \eqref{eq:V0}, a direct integration gives
\begin{equation}\label{eq:Ustar}
U_*(x)
=
\begin{cases}
x^3\bigl(16x^2-25x+10\bigr),
& 0\leq x\leq\frac12,
\\[1mm]
(1-x)^3
\bigl(16(1-x)^2-25(1-x)+10\bigr),
& \frac12\leq x\leq1.
\end{cases}
\end{equation}
By Lemma~\ref{lem:reduction}(iv) and the above characterization of the
extremal functions for \eqref{eq:Cstar}, the extremal functions for
$\lambda_3(3,1,2,1)$ are precisely the nonzero scalar multiples of
$U_*$. Moreover, \eqref{eq:Ustar} gives
\[
U_*(1-x)=U_*(x),
\qquad 0\leq x\leq1.
\]
Thus every extremal function for $\lambda_3(3,1,2,1)$ is symmetric
about $x=\frac12$.

\medskip
\noindent\textbf{Step 3: Extremal functions for
$\lambda_3(3,0,2,\infty)$.}

By \eqref{eq:lambda3}, 
\[
\lambda_3(3,0,2,\infty)
=
\sup_{U\in\mathring W_2^3(0,1)\setminus\{0\}}
\frac{\|U\|_{L^\infty(0,1)}}
{\|U'''\|_{L^2(0,1)}}.
\]
We first show that $U_*$ is an extremal function for
$\lambda_3(3,0,2,\infty)$.

Indeed, by \eqref{eq:sign-V0} and \eqref{eq:U-star},
\[
U_*'(x)
=
-120V_0(2x-1),
\qquad 0\leq x\leq1,
\]
and
\[
\operatorname{sgn}(V_0(t))
=
\operatorname{sgn}(t),
\qquad -1<t<1.
\]
Hence $U_*$ is increasing on $(0,\frac12)$ and decreasing on
$(\frac12,1)$. Together with
\[
U_*(0)=U_*(1)=0,
\]
this yields
\[
\|U_*\|_{L^\infty(0,1)}
=
U_*\left(\frac12\right)
\]
and
\begin{equation}\label{eq:equality-norms-U-star}
\begin{aligned}
\|U_*'\|_{L^1(0,1)}
&=
\int_0^{\frac12}U_*'(x)\,dx
-
\int_{\frac12}^1 U_*'(x)\,dx\\
&=
2U_*\left(\frac12\right)\\
&=
2\|U_*\|_{L^\infty(0,1)}.
\end{aligned}
\end{equation}
Since $U_*$ is an extremal function for $\lambda_3(3,1,2,1)$,
\eqref{eq:sharp-cst} gives
\[
\frac{\|U_*'\|_{L^1(0,1)}}
{\|U_*'''\|_{L^2(0,1)}}
=
\frac{1}{32\sqrt5}.
\]
Therefore, by \eqref{eq:constant-WatanabeEtAl} and
\eqref{eq:equality-norms-U-star},
\[
\frac{\|U_*\|_{L^\infty(0,1)}}
{\|U_*'''\|_{L^2(0,1)}}
=
\frac{1}{64\sqrt5}
=
\lambda_3(3,0,2,\infty).
\]
Thus $U_*$ is an extremal function for
$\lambda_3(3,0,2,\infty)$.

Let
\[
\widetilde{\mathcal P}_2
=
\operatorname{span}\{1,s,s^2\}
\subset L^2(0,1),
\]
and denote by
\[
\widetilde\Pi_2:L^2(0,1)\longrightarrow\widetilde{\mathcal P}_2
\]
the orthogonal projection onto $\widetilde{\mathcal P}_2$. For $x\in[0,1]$, set
\[
\sigma_x
=
(I-\widetilde\Pi_2)
\left(\frac{(x-\cdot)_+^2}{2}\right).
\]
A direct computation gives
\[
\|\sigma_x\|_{L^2(0,1)}^2
=
\frac{x^5(1-x)^5}{20},
\qquad 0\leq x\leq1.
\]
Consequently,
\begin{equation}\label{eq:max-norm-sigma}
\|\sigma_x\|_{L^2(0,1)}
\leq
\frac{1}{64\sqrt5},
\end{equation}
with equality if and only if $x=\frac12$.

Let $U\in\mathring W_2^3(0,1)\setminus\{0\}$ be an arbitrary
extremal function for $\lambda_3(3,0,2,\infty)$. Choose
$x_0\in(0,1)$ such that
\[
|U(x_0)|
=
\|U\|_{L^\infty(0,1)},
\]
and set 
\[
h=U'''.
\]
Then $h\neq0$ and 
\[
\frac{|U(x_0)|}
{\|h\|_{L^2(0,1)}}=\lambda_3(3,0,2,\infty)=\frac{1}{64\sqrt5}.
\]
The boundary conditions
\[
U(0)=U(1)=U'(0)=U'(1)=U''(0)=U''(1)=0
\]
and integration by parts give
\[
\int_0^1 h(s)\,ds
=
\int_0^1 sh(s)\,ds
=
\int_0^1 s^2h(s)\,ds
=
0.
\]
Hence
\[
h\perp\widetilde{\mathcal P}_2
\]
and 
\[
U(x)
=
\frac12\int_0^1(x-s)_+^2h(s)\,ds
=
\langle\sigma_x,h\rangle_{L^2(0,1)},
\qquad 0\leq x\leq1.
\]
Using \eqref{eq:max-norm-sigma}, we obtain 
\[
\begin{aligned}
\frac{1}{64\sqrt5}\|h\|_{L^2(0,1)}
&=
|U(x_0)|\\
&=
|\langle\sigma_{x_0},h\rangle_{L^2(0,1)}|\\
&\leq
\|\sigma_{x_0}\|_{L^2(0,1)}
\|h\|_{L^2(0,1)}\\
&\leq
\frac{1}{64\sqrt5}\|h\|_{L^2(0,1)}.
\end{aligned}
\]
Since $h\neq0$, equality holds throughout. Therefore, 
\[
\|\sigma_{x_0}\|_{L^2(0,1)}
=
\frac{1}{64\sqrt5},\qquad |\langle\sigma_{x_0},h\rangle_{L^2(0,1)}|=\|\sigma_{x_0}\|_{L^2(0,1)}
\|h\|_{L^2(0,1)}.
\]
Consequently, $x_0=\frac{1}{2}$ and  
\[
h=c\,\sigma_{1/2}
\]
for some $c\neq 0$.  

Applying the same argument to the extremal function $U_*$ gives
\[
U_*'''=c_*\,\sigma_{1/2}
\]
for some $c_*\neq0$. Hence
\[
U'''=c_0U_*''',\qquad c_0=\frac{c}{c_*}.  
\]
Since
\[
U-c_0U_*\in\mathring W_2^3(0,1)
\]
and
\[
(U-c_0U_*)'''=0,
\]
we obtain
\[
U=c_0U_*.
\]
Thus the extremal functions for $\lambda_3(3,0,2,\infty)$ are
precisely the nonzero scalar multiples of $U_*$. Combining this with Step~2, the families of extremal functions for
$\lambda_3(3,1,2,1)$ and $\lambda_3(3,0,2,\infty)$ coincide up to
multiplication by a nonzero constant. Moreover, 
every such extremal function is symmetric about $x=\frac12$.
This completes the proof.
 
\end{proof}

\end{document}